\documentclass[12pt]{article}
\usepackage{amsmath,amssymb,amsthm,amsfonts,amscd}
\usepackage{hyperref}
\usepackage[numeric]{amsrefs}
\usepackage{rotating} 
\input colordvi
\usepackage{graphicx}

\newtheorem{thm}[equation]{Theorem}

\newtheorem{lem}[equation]{Lemma}

\newcommand{\thmref}[1]{Theorem~\ref{#1}}

\newcommand{\lemref}[1]{Lemma~\ref{#1}}
\newcommand{\secref}[1]{Section~\ref{#1}}

\numberwithin{equation}{section}

\renewcommand\a{\alpha}
\renewcommand\b{\beta}
\newcommand\g{\gamma}
\renewcommand\d{\delta}
\newcommand\e{\varepsilon}
\renewcommand\l{\lambda}
\renewcommand\L{\Lambda}
\newcommand\D{\Delta}

\renewcommand\D{\Delta}
\newcommand\G{\Gamma}
\newcommand\f{\frac}
\newcommand\smallf[2]{{\textstyle{\frac{#1}{#2}}}}

\newcommand{\Z}{{\mathbb{Z}}}
\newcommand{\R}{{\mathbb{R}}}

\newcommand{\C}{{\mathbb{C}}}

\newcommand{\U}{{\mathbb{H}}}

\renewcommand\i{^{-1}}
\renewcommand\({\left(}
\renewcommand\){\right)}
\newcommand{\ttwo}[4]{
\(\begin{smallmatrix}{#1} & {#2}
\\ {#3} & {#4} \end{smallmatrix}\)}
\newcommand{\tthree}[9]{\(\begin{smallmatrix}{#1} & {#2} & {#3}
\\ {#4} & {#5} & {#6} \\ {#7} & {#8} & {#9} \end{smallmatrix}\)}
\newcommand{\lietwo}[4]{
\left[\begin{smallmatrix}{#1} & {#2}
\\ {#3} & {#4} \end{smallmatrix}\right]}
\newcommand{\liethree}[9]{\left[\begin{smallmatrix}{#1} & {#2} & {#3}
\\ {#4} & {#5} & {#6} \\ {#7} & {#8} & {#9} \end{smallmatrix}\right]}

\newcommand{\sgn}{\operatorname{sgn}}

\newcommand\srel[2]{\begin{smallmatrix} {#1} \\ {#2} \end{smallmatrix}}

\newcommand{\gobble}[1]{}
  \newcommand{\rangeref}[2]{%
    \ref{#1}--\afterassignment\gobble\fam 0\ref{#2}%
  }

\makeatletter
\def\imod#1{\allowbreak\mkern5mu({\operator@font mod}\,#1)}
\makeatother

\begin{document}

\title{Weights, raising and lowering operators, and $K$-types for automorphic forms on $SL(3,\R)$}

\date{February 28, 2017}

\author{Jack Buttcane\thanks{Supported by NSF grant DMS-1601919.} \\University at Buffalo \\ \tt{buttcane@buffalo.edu} \\
and \\Stephen D. Miller\thanks{Supported by NSF grant DMS-1500562.}\\
Rutgers University\\
\tt{miller@math.rutgers.edu}}

\maketitle

\begin{abstract}
 We give a fully explicit description of Lie algebra derivatives (generalizing raising and lowering operators) for representations of $SL(3,\R)$ in terms of a basis of Wigner functions.  This basis is natural from the point of view of principal series representations, as well as computations in the analytic theory of automorphic forms (e.g., with Whittaker functions).  The method is based on the Clebsch-Gordan multiplication rule for Wigner functions, and applies to other Lie groups whose maximal compact subgroup is isogenous to a product of $SU(2)$ and $U(1)$ factors.    As an application, we give a complete and explicit description of the $K$-type  structure of  certain cohomological representations.
\end{abstract}

\section{Introduction}\label{sec:intro}

One of the most prominent features of the classical  theory of holomorphic modular forms is the {\em weight}, i.e., the integer (or sometimes fraction) $k$ appearing in the transformation law
\begin{equation}\label{weightk}
  f(z) \ \ = \ \ (cz+d)^{-k}\,f\(\f{az+b}{cz+d}\),
\end{equation}
where $f$ is a modular   form and $\ttwo abcd$ is an element of its discrete automorphy group.
Maass developed a parallel theory for non-holomorphic forms of various weights, including his well-known explicit operators which increase and lower the weight by 2.
  Gelfand-Graev (see \cite{Gelfand-Graev}) gave a reformulation of modular and Maass forms as functions on the group $SL(2,\R)$,
 in which (\ref{weightk}) is interpreted representation-theoretically as an $SO(2)$-equivariance condition (see (\ref{SO2isotypic1})), and Maass's raising and lowering operators naturally correspond to Lie algebra derivatives (see (\ref{sl2updownaction}) and  (\ref{sl2diffops3})).  Those same Lie algebra derivatives are   decisive tools in Bargmann's work \cite{Bargmann} describing the irreducible unitary representations of $SL(2,\R)$ via their $K$-types.

The goal of this paper is to describe a parallel, explicit theory of raising and lowering operators for automorphic forms on   $SL(3,\R)$.  The papers \cite{howe,miyazaki} give a description of Lie algebra derivatives in principal series representations for $SL(3,\R)$.  In this paper we instead use a very explicit basis provided by Wigner functions \cite{bied} on the maximal compact subgroup $K=SO(3)$, from which one can then study precise topics in the analytic theory of automorphic forms (such as Whittaker functions, pursued in \cite{buttcane}).  Indeed, our motivation is that restriction of a $K$-isotypic automorphic form on $SL(3,\R)$ to $K$ is itself a sum of Wigner functions, which justifies the importance of this basis.

As do the methods of \cite{howe,miyazaki},
our method  also allows the computation of the
composition series of principal series representations of $SL(3,\R)$, as well as the action of intertwining operators on $K$-types.   Moreover, it applies (with straightforward modifications) to many  other real reductive Lie groups, in particular  those whose maximal compact subgroup is isogenous to a product of $SU(2)$'s and tori (such as $G_2$ and $SO(4,4)$, pursued in  \cite{Zhang,GLMZ}), and indeed to any group satisfying a sufficiently-explicit analog of the  Clebsch-Gordan multiplication formula (\ref{CG1}).

We will use the notation $G$ to denote $SL(r,\R)$, where $r$ will equal 3 except in \secref{sec:SL2R} (where $r=2$), and $\frak g$ to denote the complex Lie algebra $\frak{sl}_r$ (consisting of all traceless $r\times r$ complex matrices).  We write $N\subset G$ for the subgroup of unit upper triangular  $r\times r$  matrices, which is a  maximal unipotent subgroup of $G$, and $\frak n$ for its complexified Lie algebra (which consists of all strictly upper triangular complex  $r\times r$ matrices).  The subgroup $A\subset G$ consists of all nonsingular,  diagonal  $r\times r$  matrices with positive entries and determinant 1; it is  the connected component of a maximal abelian subgroup of $G$. The complexified Lie algebra $\frak a$ of $A$ consists of all traceless diagonal $r\times r$  matrices with complex entries.  Finally, $K=SO(r)$ is a maximal compact subgroup of $G$, and its complexified Lie algebra $\frak k$ consists of all antisymmetric complex $r\times r$ matrices.  The Iwasawa decomposition asserts that the map $(n,a,k)\mapsto nak$ from $N\times A \times K\rightarrow G$ is a diffeomorphism; at the level of Lie algebras, $\frak g={\frak{n}}\oplus{\frak{a}}\oplus{\frak{k}}$.

Our main result gives the explicit action of the Lie algebra on the basis of Wigner functions in a principal series representation (see (\ref{basisforlisotypicprincseries})).  The action of $\frak k$ is classical and described in (\ref{so3differentiation}); the following result describes the action of a basis of the five-dimensional complement $\frak p=\{X\in{\frak{g}}|X^t=X\}$ of $\frak k$ in $\frak g$:
\begin{thm}\label{thm:intro}
Let $Z_{-2}$, $Z_{-1}$, $Z_{0}$, $Z_1$, and $Z_2$ be as defined in (\ref{sl3liealg}), $\L^{(k)}_{j}(\l,\ell,m_1)$ be as defined in (\ref{Lambda}), and let $q_{k,j}(\ell,m)$ denote the Clebsch-Gordan coefficient $\langle
2\,k\,\ell\,m|(\ell+j)\,(k+m)
\rangle$  (see (\ref{CG3}) for an explicit description in the cases of relevance).  Set $c_{-2}=c_2=1$ and $c_0=\sqrt{\f 23}$. Let
$V_{\lambda,\delta}$ be a principal series representation of $SL(3,\R)$ (see (\ref{princsl3}))
and let $v_{\ell,m_1,m_2}$ denote its elements defined in (\ref{basisforlisotypicprincseries}).
Then
\begin{multline}\label{introthm1}
   \pi(Z_n)v_{\ell,m_1,m_2}  \ \ = \\  \sum_{\srel{-2\le j \le 2}{k\,\in\,\{-2,0,2\}}} c_k \,  q_{k,j}(\ell,m_1)\, q_{n,j}(\ell,m_2) \, \Lambda^{(k)}_j(\lambda,\ell,m_1) \, v_{\ell+j,m_1+k,m_2+n}\,.
\end{multline}
\end{thm}
\noindent
The papers \cite{howe,miyazaki} give comparable results for differently-presented bases.

Section~\ref{sec:SL2R} contains a review of Lie algebra derivatives, $K$-types, and composition series of principal representations for the group $SL(2,\R)$.  Section~\ref{sec:wignerso3} gives background on Wigner functions for $SO(3)$, and Section~\ref{sec:princseries} describes principal series representations of $SL(3,\R)$ in terms of Wigner functions.   Theorem~\ref{thm:intro}  is proved in Section~\ref{sec:gKmodulestructure}, along with a description of the operators $U_j$ (\ref{Ujdef}), which are somewhat analogous to raising and lowering operators.  As an application, in Section~\ref{sec:examples} we describe the composition series of some principal series representations relevant to automorphic forms, in terms of Wigner functions.  Finally, Section~\ref{sec:diffopformulas} gives formulas for $\pi(Z_{\pm 2})$ and $\pi(Z_0)$ from Theorem~\ref{thm:intro} in terms of differential operators on the symmetric space $SL(3,\R)/SO(3,\R)$.

We wish to thank Jeffrey Adams, Michael B. Green, Roger Howe, Anthony Knapp, Gregory Moore, Siddhartha Sahi, Peter Sarnak, Wilfried Schmid, Pierre Vanhove, Ramarathnam Venkatesan, David Vogan, Zhuohui Zhang, and Greg Zuckerman for their helpful discussions.

\section{$SL(2,\R)$ background}\label{sec:SL2R}

This section contains a summary of the types of results we prove for $SL(3,\R)$, but in the simpler and classical context of $SL(2,\R)$.
For any function  $f(z)$   on the complex upper half plane $\U$, Gelfand-Graev (see \cite{Gelfand-Graev})  associated a function $\phi$ on $G=SL(2,\R)$ by the formula
\begin{equation}\label{gglifting}
    \phi\ttwo abcd \ \ = \ \ (ci+d)^{-k}\,f\(\f{ai+b}{ci+d}\).
\end{equation}
If $f$ satisfies (\ref{weightk}) for all $\ttwo abcd$ lying in a discrete subgroup $\G\subset G$, then $\phi(\gamma g)=\phi(g)$ for all $\g\in \G$.  Finally, since $K=SO(2,\R)$ fixes $z=i$  under its action by fractional linear transformations on $\U$,   the weight  condition (\ref{weightk}) for $f$ becomes the $K$-isotypic condition
\begin{equation}\label{SO2isotypic1}
  \phi(gk) \ \ = \ \ \chi(k)\,\phi(g)\,,
\end{equation}
for $\phi$, where $\chi$ is a character of the group $K=SO(2,\R)$.\footnote{More precisely, $\chi$ is the character $\chi_{-k}$ defined in (\ref{charactersofSO2}).}

Thus Gelfand-Graev converted the study of modular forms to the study of functions on $\G\backslash G$ which transform on the right by a character of the maximal compact subgroup $K$.
By a standard reduction, it suffices to study functions $\phi$ which lie in an irreducible automorphic representation, in particular a space that is invariant  under the right translation action $[\rho(g)\phi](x)=\phi(xg)$.  More precisely, one assumes the existence of an irreducible representation $(\pi,V)$ of $G$ and an embedding $j$ from $V$ into a space of functions on $\Gamma\backslash G$ which intertwines the two representations in the sense that
\begin{equation}\label{jandpi}
j(\pi(g)v) \ \ = \ \ \rho(g)j(v)\,.
\end{equation}
Frequently an $L^2$-condition is imposed on $j(v)(g)$, though that will not be necessary here.

Under the assumption that $j(v)=\phi$, condition (\ref{SO2isotypic1}) can be elegantly restated as
\begin{equation}\label{SO2isotypic2}
  \pi(k)v \ \  =  \ \ \chi(k)v\,,
\end{equation}
that is, $v$ is an isotypic vector for the character $\chi$ of $K=SO(2,\R)$.  The representation theory of compact Lie groups was completely determined by Weyl, and in the present context of $K=SO(2,\R)$ the irreducibles are simply the characters
\begin{equation}\label{charactersofSO2}
  \chi_\ell(k_\theta) \ \ = \ \ e^{i\ell\theta}\, , \ \ \ k_\theta \ = \ \ttwo{\cos\theta}{-\sin\theta}{\sin\theta}{\cos\theta}\ \in \ SO(2,\R)\,,
\end{equation}
as $\ell$ ranges over $\Z$.

 Writing
 \begin{equation}\label{Vell}
 V_\ell \ \ = \ \ \{v\in V\ | \ \pi(k)v\, =\, \chi_\ell(k)v\}\,,
\end{equation}
the direct sum $\oplus_{\ell\in\Z}V_\ell$ forms the Harish-Chandra module of {\em $K$-finite vectors} for $\pi$, i.e., those vectors whose $K$-translates span  a finite-dimensional subspace.  In general, the Harish-Chandra module can be defined in terms of the decomposition of $\pi$'s restriction to $K$ into irreducible representations.
  In this particular situation, $K$-finite vectors correspond to trigonometric polynomials, and $K$-isotypic vectors correspond to trigonometric monomials.  The full representation $V$ is a completion of this space in terms of classical Fourier series.  For example, its smooth vectors $V^\infty$ (those vectors for which the map $g\mapsto \pi(g)v$ is a smooth function on the manifold $G$) correspond to Fourier series whose coefficients decay faster than the reciprocal of any polynomial.

  By definition, the subspace $V^\infty$ is preserved by Lie algebra derivatives
  \begin{equation}\label{liederiv}
    \pi(X)v \ \ := \ \ \lim_{t\,\rightarrow \,0} \f{\pi(e^{tX})v-v}{t}\,, \ \ X \ \in \ \frak g_\R\,,
  \end{equation}
  where $\frak g_\R$ is the Lie algebra of $G$.  This definition extends to the complexified Lie algebra $\frak g$ through the linearity rule $\pi(X_1+iX_2)=\pi(X_1)+i\pi(X_2)$.
  If $\phi=j(v)$ is a smooth automorphic function corresponding to $v\in V^\infty$, then $j(\pi(X)v)$
  is equal to $(\rho(X)\phi)(g) := \left.\f{d}{dt}\right|_{t=0}\phi(ge^{tX})$; as before, $\rho(X)$ is initially defined only for $X\in\frak g_\R$ and is then extended to $X\in \frak g$ by linearity.

    These Lie algebra derivatives occur in various {\em ad hoc} guises in the classical theory of automorphic functions, including Maass's raising and lowering operators.  Such derivatives satisfy various relations with each other, which can often  be more clearly seen by doing computations in a suitably chosen model for the representation $(\pi,V)$.  It is well-known that every representation  of $SL(2,\R)$ is a subspace of some principal series representation $(\pi_{\nu,\e},V_{\nu,\e})$, where
\begin{equation}\label{Vnudef}
\gathered
  V_{\nu,\epsilon} \ \ = \ \ \left\{ f:G\rightarrow\C\ \mid \ f\(\ttwo{a}{b}{0}{d}g\) = |\smallf{a}{d}|^{1/2+\nu} \sgn(a)^\e f(g)  \right\}, \\
  [\pi_{\nu,\e}(g)f](h) \ \ = \ \ f(hg)\,,
  \endgathered
\end{equation}
 $\nu\in \C$, and $\epsilon \in \Z/2\Z$.  We shall thus use subspaces of (\ref{Vnudef}) as convenient models for arbitrary representations.

 Recall the {\em Iwasawa decomposition} $G=NAK$, where $N=\{\ttwo 1x01|x\in\R\}$ and $A=\{\ttwo{a}{0}{0}{a\i}|a>0\}$:~each element $g\in G$ has a  decomposition
\begin{equation}\label{expliiwasawasl2a}
    g \ \ = \ \     \ttwo 1x01 \ttwo{a}{0}{0}{a\i} k_\theta\,,
\end{equation}
with $x$, $a$, and $\theta$ determined uniquely by $g$.  It follows that any function $f$ in (\ref{Vnudef}) is completely determined by its restriction to $K$; since $k_\pi=\ttwo{-1}00{-1}$, this restriction must satisfy the parity condition
\begin{equation}\label{sl2paritycondition}
    f(k_\theta) \ \ = \ \ (-1)^\e\, f(k_{\theta+\pi})\,.
\end{equation}
   Conversely, any function on $K$ satisfying (\ref{sl2paritycondition}) extends to an element of $V_{\nu,\e}$, for any $\nu\in \C$.  The $K$-isotypic subspace $V_\ell$ therefore vanishes unless $\ell\equiv \e\imod 2$.
When $V$ is equal to a full, irreducible principal series representation $V_{\nu,\e}$, $V_{\ell}$ is one-dimensional for $\ell\equiv \e\imod 2$ and consists of all complex multiples of $v_\ell$, the element of $V_{\nu,\e}$ whose restriction to $K$ is $\chi_\ell$.
In terms of the Lie algebra, membership in $V_{\ell}$ is characterized as
\begin{equation}\label{pikgenonvl}
  v\,\in\,V_{\ell} \ \ \Longleftrightarrow \ \   \pi\!\(\lietwo{0}{-1}{1}{0}\)v  \,=\, i \ell v\,,
\end{equation}
 since $\lietwo{0}{-1}{1}{0}\in \frak g$ is the infinitesimal generator of $K$, i.e., $\exp(t\lietwo{0}{-1}{1}{0})=k_{t}$.\footnote{In order to avoid confusion of whether $\pi$ or $\rho$ is acting via group translation in $G$ or by derivation in $\frak g$, we will use the traditional matrix parenthesis notation $\ttwo\cdot\cdot\cdot\cdot$ to signify elements of $G$ and the bracket notation $\lietwo\cdot\cdot\cdot\cdot$ to signify elements of $\frak g$.}
 The following well-known result computes the action of a basis of $\frak g$ on the $v_\ell$ inside $V_{\nu,\e}$:
\begin{lem}\label{lem:sl2Liealgderiv}  With the above notation,
\begin{equation}\label{sl2updownaction}
\aligned
 \pi\(\lietwo{1}{i}{i}{-1}\) v_\ell \ \ & = \ \ (2\nu+1-\ell)\,v_{\ell-2}\,, \\
  \pi\!\(\lietwo{0}{-1}{1}{0}\)v_\ell   \ \ & =\ \  i \,  \ell\, v_\ell\,,\\
    \text{and} \ \ \ \pi\(\lietwo{1}{-i}{-i}{-1}\) v_\ell \ \ & = \ \ (2\nu+1+\ell)\,v_{\ell+2}\,.
\endaligned
\end{equation}
\end{lem}
\noindent
Our main result
theorem~\ref{thm:main} is a generalization of \lemref{lem:sl2Liealgderiv} to $SL(3,\R)$.
Formulas (\ref{sl2updownaction}) are collectively equivalent to the three formulas
\begin{equation}\label{sl2Liealgderivs}
  \aligned
  \pi\(\lietwo{0}{1}{0}{0}\)v_\ell \ \ & = \ \ \smallf{i(\ell-1-2\nu)}{4}\,v_{\ell-2} \ - \ \smallf{i \ell}{2} \,v_{\ell} \ + \ \smallf{i(\ell+1+2\nu)}{4}\,v_{\ell+2}\,, \\
   \pi\(\lietwo{1}{0}{0}{-1}\)v_\ell \ \ &  = \ \ -\,\smallf{(\ell-1-2\nu)}{2}\,v_{\ell-2}  \ \  \ \ \ +  \ \ \ \ \
    \ \smallf{ (\ell+1+2\nu)}{2}\,v_{\ell+2}\,, \\ \text{and} \ \ \
    \pi\(\lietwo0010\)v_\ell \ \ &  = \ \ \smallf{i(\ell-1-2\nu)}{4}\,v_{\ell-2} \ + \ \smallf{i \ell}{2}\, v_{\ell} \ + \ \smallf{i(\ell+1+2\nu)}{4}\,v_{\ell+2} \\
  \endaligned
\end{equation}
for the action of $\frak g=\frak{sl}_2(\C)$ under its usual basis.  However, (\ref{sl2updownaction}) is simpler because its three Lie algebra elements   diagonalize the adjoint (conjugation) action $Ad(g):X\mapsto gXg\i$ of $K$ on $\frak g$:
\begin{equation}\label{sl2adjointdiag}
  \aligned
      Ad(k)\lietwo{1}{i}{i}{-1} \ \ & = \ \ \chi_{-2}(k)\,\lietwo{1}{i}{i}{-1} , \\
  Ad(k)\lietwo{0}{-1}{1}{0} \ \ & = \ \  \lietwo{0}{-1}{1}{0} , \\
\text{and} \ \ \   Ad(k)\lietwo{1}{-i}{-i}{-1} \ \ & = \ \ \chi_{2}(k)\,\lietwo{1}{-i}{-i}{-1} . \\
  \endaligned
\end{equation}
Indeed, the second of these formulas is obvious because $\lietwo{0}{-1}{1}{0}$ is the infinitesimal generator of the abelian group $K$, while the first and third formulas are equivalent under complex conjugation;~both can be seen either by direct  calculation, or more simply verifying the equivalent Lie bracket formulation
\begin{equation}\label{equivlie}
  \big[ \lietwo0{-1}10,\lietwo1ii{-1}   \big] \ \ = \ \  \lietwo0{-1}10\lietwo1ii{-1}  \, - \, \lietwo1ii{-1}\lietwo0{-1}10 \ \ = \ \ -2\,i\, \lietwo1ii{-1}.
\end{equation}
 Incidentally, it follows from (\ref{equivlie}) that $\pi\(\lietwo0{-1}10\)\pi\(\lietwo1ii{-1}\)v_\ell=-2i\pi\(\lietwo1ii{-1}\)v_\ell+\pi\(\lietwo1ii{-1}\)\pi\(\lietwo0{-1}10\)v_\ell$, which equals $i(\ell-2)\pi\(\lietwo1ii{-1}\)v_\ell$ by (\ref{pikgenonvl}). A second application of (\ref{pikgenonvl}) thus shows that $\pi\(\lietwo1ii{-1}\)v_\ell\in V_{\ell-2}$, and hence a multiple of $v_{\ell-2}$; the first formula in (\ref{sl2updownaction}) is more precise in that it determines the exact multiple.  Although Lemma~\ref{lem:sl2Liealgderiv} is well-known, we nevertheless include a proof in order to motivate some of our later calculations:

\begin{proof}[Proof of \lemref{lem:sl2Liealgderiv}]
If $X\in \frak g=\frak{sl}_2(\R)$, then by definition
\begin{equation}\label{lem:sl2Liealgderivpf1}
    (\pi(X)v_{\ell})(k_\theta) \ \ = \ \ \left.\f{d}{dt}\right|_{t=0} v_{\ell}(k_\theta e^{tX})
    \ \ = \ \ \left.\f{d}{dt}\right|_{t=0} v_{\ell}( e^{t k_\theta Xk_{-\theta}}k_\theta)\,.
\end{equation}
Expand the Lie algebra element $k_\theta Xk_{-\theta}$ as a linear combination
\begin{equation}\label{sl2Liealgderivpf2}
    k_\theta Xk_{-\theta}  \ \ = \ \ c_E(X,\theta)\lietwo0100 \ + \
    c_H(X,\theta)\lietwo{1}{0}{0}{-1}\ + \
    c_Y(X,\theta)\lietwo{0}{-1}{1}{0},
\end{equation}
so that
(\ref{lem:sl2Liealgderivpf1}) is equal to
\begin{multline}\label{sl2Liealgderivpf4}
    (\pi(X)v_{\ell})(k_\theta) \ \ = \ \
    c_E(X,\theta) \left.\f{d}{dt}\right|_{t=0} v_{\ell}( \ttwo1t01 k_\theta) \ + \\
    c_H(X,\theta) \left.\f{d}{dt}\right|_{t=0} v_{\ell}( \ttwo{e^t}{0}{0}{e^{-t}} k_\theta) \ + \
    c_Y(X,\theta) \left.\f{d}{dt}\right|_{t=0} v_{\ell}(k_{t}k_\theta)\,.
\end{multline}
By definition (\ref{Vnudef}), $v_\ell(\ttwo1t01 k_\theta)$ is independent of $t$ while
 $v_\ell(\ttwo{e^t}{0}{0}{e^{-t}} k_\theta)=e^{(2\nu+1)t}v_\ell( k_\theta)$, so (\ref{sl2Liealgderivpf4}) equals
 \begin{equation}\label{sl2Liealgderivpf5}
 \aligned
     (\pi(X)v_{\ell})(k_\theta) \ \ & =  \ \
      c_H(X,\theta)(2\nu+1) v_\ell( k_\theta) \ + \ c_Y(X,\theta)\smallf{d}{d\theta} v_\ell( k_\theta) \\
      & =  \ \
      c_H(X,\theta)(2\nu+1) v_\ell( k_\theta) \ + \ c_Y(X,\theta)i\ell v_\ell( k_\theta)\,. \\
  \endaligned
 \end{equation}
Formula (\ref{sl2Liealgderivpf5}) remains valid for any $X$ in the complexification $\frak g$, and can be used to calculate any of the identities in (\ref{sl2updownaction}) and (\ref{sl2Liealgderivs}).  The second formula in (\ref{sl2updownaction}) was shown in (\ref{pikgenonvl}), while the first and third are equivalent. We thus consider $X=\lietwo{1}{-i}{-i}{-1}$ and calculate  that  $c_H(X,\theta)=e^{2i\theta}$  and $c_Y(X,\theta)=-ie^{2i\theta}$, so that  formula (\ref{sl2Liealgderivpf5}) specializes to $e^{2i\theta}(2\nu+1)e^{i\ell \theta}-ie^{2i\theta}i\ell e^{i\ell\theta}= (2\nu+1+\ell)e^{i(\ell+2)\theta}$ as claimed in the third equation in (\ref{sl2updownaction}).
%
%
%
%

\end{proof}

The operators $\pi\(\lietwo{1}{-i}{-i}{-1}\)$ and $\pi\(\lietwo{1}{i}{i}{-1}\)$ in (\ref{sl2updownaction})  are the Maass raising and lowering operators, respectively; they are expressed in terms of differential operators in (\ref{sl2diffops3}).
 Note that the coefficient $2\nu+1\pm\ell$ can vanish when $\nu\in \f12\Z$.
 In such a case, the representation $V_{\nu,\e}$ is reducible, but otherwise it is not (because appropriate compositions of the raising and lowering operators map any isotypic vector to any other as one goes up and down the ladder of $K$-types).  Indeed,  the following well-known characterization of irreducible representations of $SL(2,\R)$ and their $K$-type structure  can read off from (\ref{sl2updownaction}):

\begin{thm}\label{thm:irrepsofSL2R}(Bargmann \cite{Bargmann})

1) The representation $V_{\nu,\e}$ is irreducible if and only if $\nu\notin \f{\e+1}{2}+\Z$.

2) If $k\ge 2$ is an integer congruent to $\e \imod 2$,
 then $V_{-\f{k-1}{2},\e}$ contains the  $(k-1)$-dimensional irreducible representation\footnote{There is a unique such representation up to equivalence.  It is given by the action of $SL(2,\R)$ on homogeneous polynomials of degree $k-2$ in 2 variables.} of $SL(2,\R)$, spanned by the  basis $\{v_{2-k},v_{4-k},v_{6-k},\ldots,v_{k-4},v_{k-2}\}$.

3)  If $k\ge 2$ is an integer congruent to $\e \imod 2$,   then $V_{\f{k-1}{2},\e}$ contains the direct sum of two irreducible representations of $SL(2,\R)$:~a holomorphic discrete series representation $D_k$ with basis $\{v_k,v_{k+2},v_{k+4},\ldots\}$, and its antiholomorphic conjugate $\overline{D_k}$ with basis $\{v_{-k},v_{-k-2},v_{-k-4},\ldots\}$.

4) $V_{\nu,\e}$ and $V_{-\nu,\e}$ are dual to each other, hence the quotient of $V_{-\frac{k-1}{2},\e}$ by its finite-dimensional subrepresentation in case 2) is the direct sum of $D_k$ and $\overline{D_k}$.  Likewise, the quotient $V_{\f{k-1}{2},\e}/(D_k\oplus \overline{D_k})$   is the $(k-1)$-dimensional representation.

\end{thm}

To summarize, we have described the irreducible representations of $K=SO(2,\R)$, seen how functions on the circle  of a given parity decompose in terms of them, and calculated the action of raising and lowering operators on them.
  In the next few sections we shall extend some of these results to $G=SL(3,\R)$ and $K=SO(3,\R)$, which is much more complicated because $K$ is no longer an abelian group.

\section{Irreducible representations of $SO(3,\R)$ and Wigner functions}\label{sec:wignerso3}

The group $K=SO(3,\R)$ is isomorphic to $SU(2)/\{\pm I\}$, and so its irreducible representations are precisely those of $SU(2)$ which are trivial on $-I$.  By Weyl's Unitarian Trick (which in this case is actually due to  Hurwitz \cite{hurwitz}), the irreducible representations of $SU(2)$ are restrictions of the irreducible finite-dimensional representations of $SL(2,\C)$.  Those, in turn, are   furnished by the action of $SL(2,\C)$ on the $(n+1)$-dimensional vector space of degree $n$ homogeneous polynomials in two variables (on which $SL(2,\C)$ acts by linear transformations of the variables), for $n\ge 0$.  This action is trivial on $-I$ if and only if $n$ is even, in which case it factors through $SO(3,\R)$.

Thus the irreducible representations $V_\ell$  of $SO(3,\R)$  are indexed by an integer $\ell\ge 0$ (commonly interpreted as  angular momentum) and have dimension $2\ell+1$.  A basis for $V_\ell$  can be given by the isotypic vectors for any embedded $SO(2,\R)\subset SO(3,\R)$, transforming by the characters $\chi_{-\ell},\chi_{1-\ell},\ldots,\chi_\ell$ from (\ref{charactersofSO2}).

According to the Peter-Weyl theorem, $C^\infty(SO(3,\R))$ has an orthonormal decomposition into copies of the representations $V_{\ell}$ for $\ell=0,1,2,\ldots$, each occurring with multiplicity $\dim V_{\ell}=2\ell+1$.
Wigner gave an explicit  realization of  $V_\ell$ whose matrix entries give a convenient, explicit basis for these $2\ell+1$ copies.
It is most cleanly stated in terms of Euler angles for   matrices $k\in SO(3,\R)$, which always have at least one factorization of the form
\begin{multline}\label{eulerangles}
  k(\alpha,\beta,\gamma) \ \ = \ \ \tthree{\cos \a}{-\sin \a}{0}{\sin \a}{\cos \a}{0}{0}{0}{1}
  \tthree{1}{0}{0}{0}{\cos\b}{-\sin\b}{0}{\sin\b}{\cos\b}
  \tthree{\cos \g}{-\sin \g}{0}{\sin \g}{\cos \g}{0}{0}{0}{1} , \\
\a,\g \,\in\,\R/(2\pi\Z),  \ \ 0 \le \b \le \pi\,.
\end{multline}
 Wigner functions are first diagonalized according to the respective left and right actions of the $SO(2,\R)$-subgroup parameterized by $\a$ and $\g$, and are indexed by integers $\ell$, $m_1$, and $m_2$ satisfying $-\ell\le m_1,m_2 \le \ell$; they  are given by the formula\footnote{Some authors use   slightly different notation, denoting $d^{\ell}_{m_1,m_2}(\cos\b)$ simply by $d^{\ell}_{m_1,m_2}(\beta)$.  Others such as \cite{bied}   flip the signs of $m_1$ and $m_2$.  }
\begin{multline}\label{WignerDformula}
   D^{\ell}_{m_1,m_2}(k(\a,\b,\g)) \ \ = \ \ (-1)^{\ell+m_2}\, e^{i (m_1 \a+m_2\g)}\sqrt{\smallf{(\ell+m_1)!(\ell-m_1)!}{(\ell+m_2)!(\ell-m_2)!}} \ \times \\
   \sum_{r=\max(0,m_1+m_2)}^{\min(\ell+m_1,\ell+m_2)}(-1)^{r} \(\srel{\ell+m_2}{r}\)\(\srel{\ell-m_2}{\ell+m_1-r}\)
   \cos(\smallf{\b}2)^{2r-m_1-m_2} \sin(\smallf{\b}2)^{2\ell+m_1+m_2-2r} 
\end{multline}
\cite[(3.65)]{bied},
which can easily be derived from the matrix coefficients of the $(2\ell+1)$-dimensional representation of $SL(2,\C)$ mentioned above.  They can be rewritten as
\begin{equation}\label{WignerDdef}
  D^{\ell}_{m_1,m_2}(k(\a,\b,\g)) \ \ = \ \ e^{i m_1 \a}\,d^{\ell}_{m_1,m_2}(\cos \b)\,e^{i m_2 \g}\ ,  \ \ \ -\ell\,\le\,m_1,m_2\,\le\,\ell\,,
\end{equation}
where
\begin{multline}\label{littleddef}
  d^{\ell}_{m_1,m_2}(x) \ \ = \\ (-1)^{\ell-m_1}2^{-\ell}\sqrt{\smallf{(\ell+m_2)!(1-x)^{m_1-m_2}}{(\ell-m_2)!
  (\ell+m_1)!(\ell-m_1)!(1+x)^{m_1+m_2}}}
(\smallf{d}{dx})^{\ell-m_2}(1-x)^{\ell-m_1}(1+x)^{\ell+m_1}
\end{multline}
\cite[(3.72)]{bied}.
Each $D^{\ell}_{m_1,m_2}$ is an isotypic vector for the embedded $SO(2,\R)$ corresponding to the Euler angle $\g$. The full $(2\ell+1)\times (2\ell+1)$ matrix $(D^\ell_{m_1,m_2}(k))_{-\ell\le m_1,m_2\le \ell}$ furnishes an explicit representation of $SO(3)\rightarrow GL(2\ell+1,\C)$.

The Wigner functions $D^{\ell}_{m_1,m_2}$ form an orthogonal basis for the smooth functions on $K=SO(3,\R)$ under the inner product given by integration over $K$:~more precisely,
\begin{equation}\label{innerproduct}
  \int_{K}D^{\ell}_{m_1,m_2}(k)\, \overline{D^{\ell'}_{m_1',m_2'}(k)}\,dk \ \ = \ \ \left\{
                                                                           \begin{array}{ll}
                                                                             \f{1}{2\ell+1}, & \ell=\ell',m_1=m_1',\text{ and }m_2=m_2'; \\
                                                                            0, & \hbox{otherwise.}
                                                                           \end{array}
                                                                         \right.
\end{equation}
where  $dk$ is the Haar measure which assigns $K$ volume 1 \cite[(3.137)]{bied}.

The left transformation law $D^{\ell}_{m_1,m_2}(k(\a,0,0)\,k) =e^{im_1\a}D^{\ell}_{m_1,m_2}(k)$ is unchanged under  right translation by $K$, so for any fixed $\ell$ and $m_1$ the span of
  $\{D^{\ell}_{m_1,m_2}|-\ell\le m_2 \le \ell\}$  is an irreducible representation of $K$ equivalent to $V_\ell$.  The $2\ell+1$ copies stipulated by the Peter-Weyl theorem are then furnished by the possible choices of $-\ell\le m_1\le \ell$.

In our applications, it will be important to express the product of two Wigner functions as an explicit linear combination of Wigner functions using the Clebsch-Gordan multiplication formula
\begin{multline}\label{CG1}
  D^{\ell}_{m_1,m_2}\,D^{\ell'}_{m_1',m_2'} \ \ = \\ \sum_{\ell''} \langle \ell\,m_1\,\ell'\,m_1'\,|\,\ell''\,(m_1+m_1') \rangle
\langle \ell\,m_2\,\ell'\,m_2'\,|\,\ell''\,(m_2+m_2') \rangle D^{\ell''}_{m_1+m_1',m_2+m_2'}\,,
\end{multline}
where $\langle \cdot \cdot \cdot \cdot | \cdot \cdot \rangle$ are Clebsch-Gordan coefficients (see  \cite[(3.189)]{bied}).  Although the Clebsch-Gordan coefficients are somewhat messy to define in general, we shall only require them when $\ell'=2$, in which case the terms in the sum vanish unless $|\ell-2| \le \ell'' \le \ell+2$ (this   condition is known as ``triangularity'' -- see \cite[(3.191)]{bied}). Write
\begin{equation}\label{CG2}
\aligned
  q_{k,j}(\ell,m)  \ \ := & \ \  \langle 2 \, k \, \ell \, m | (\ell+j) \, (k+m)\rangle \, ,
\endaligned
\end{equation}
which by definition vanishes unless $|k|,|j|\le 2$, $|m|\le \ell$,   $|k+m|\le \ell+j$,  and $|\ell-2| \le \ell+j$ (corresponding to when the Wigner functions in (\ref{CG1}) are defined and  the triangularity condition holds).
The  values for indices obeying these conditions are
explicitly given as
\begin{equation}\label{CG3}
  \aligned
  	q_{k,-2}(\ell,m) \ \  &  = \ \  \smallf{(-1)^k \sqrt{6(\ell-m)!(\ell+m)!}}{\sqrt{\ell(\ell-1)(2\ell-1)(2\ell+1)(2-k)!(2+k)!(\ell-k-m-2)!(\ell+k+m-2)!}} \\
	q_{k,-1}(\ell,m) \ \ &  = \ \  \smallf{(-1)^k (k+\ell k+2m)\sqrt{3(\ell-m)!(\ell+m)!}}{\sqrt{\ell(\ell-1)(\ell+1)(2\ell+1)(2-k)!(2+k)!(\ell-k-m-1)!(\ell+k+m-1)!}} \\
	q_{k,0}(\ell,m)  \ \ & = \ \  \smallf{(-1)^k (2\ell^2(k^2-1)+\ell(5k^2+6km-2)+3(k^2+3km+2m^2))
\sqrt{(\ell-m)!(\ell+m)!}}{\sqrt{\ell(\ell+1)(2\ell-1)(2\ell+3)(2-k)!(2+k)!(\ell-k-m)!(\ell+k+m)!}} \\
	q_{k,1}(\ell,m) \ \ &  = \ \  \smallf{(\ell k-2m) \sqrt{3(\ell-k-m+1)!(\ell+k+m+1)!}}{\sqrt{\ell(\ell+1)(\ell+2)(2\ell+1)(2-k)!(2+k)!(\ell-m)!(\ell+m)!}} \\
	q_{k,2}(\ell,m)  \ \ & = \ \  \smallf{\sqrt{6(\ell-k-m+2)!(\ell+k+m+2)!}}{\sqrt{(\ell+1)(\ell+2)(2\ell+1)(2\ell+3)(2-k)!(2+k)!(\ell-m)!(\ell+m)!}}
  \endaligned
\end{equation}
(see \cite[Table~27.9.4]{AS} or \cite[p.~637]{bied}).
The Clebsch-Gordan coefficients satisfy the   relation
\begin{multline}\label{CG4}
	  \sqrt{\tfrac{2}{3}} \left(\ell j+\smallf{j(j+1)}{2}-3\right) q_{0,j}(\ell,m) \ \ =\ \
	\sqrt{(\ell-m)(\ell+1+m)}\, q_{-1,j}(\ell,m+1) \\ + \ \sqrt{(\ell+m)(\ell+1-m)}\, q_{1,j}(\ell,m-1)\,,
\end{multline}
which follows from \cite[34.3.8 and 34.3.14]{NIST} or direct computation from (\ref{CG3}).  As a consequence of this and (\ref{CG1}),
\begin{equation}\label{CG5}
\aligned
	&\sqrt{(\ell-m_1)(\ell+1+m_1)} D^2_{-1,n} D^\ell_{m_1+1,m_2} +\sqrt{(\ell+m_1)(\ell+1-m_1)} D^2_{1,n} D^\ell_{m_1-1,m_2} \\
	& \ \ \ \ \ \ \ \ = \ \  \sum_{-2 \le j \le 2} \Bigl( \sqrt{(\ell-m_1)(\ell+1+m_1)} \, q_{-1,j}(\ell,m_1+1) \\
	& \ \ \ \ \  \ \ \ \ \  \  \qquad + \ \sqrt{(\ell+m_1)(\ell+1-m_1)} \,q_{1,j}(\ell,m_1-1) \Bigr) q_{i,j}(\ell,m_2) \, D^{\ell+j}_{m_1,m_2+n} \\
	&\ \ \ \ \ \ \ \ =  \ \ \sqrt{\smallf 23} \sum_{-2 \le j \le 2} \(j \ell +\smallf{j(j+1)}{2}-3\) q_{0,j}(\ell,m_1) \, q_{i,j}(\ell,m_2)\,  D^{\ell+j}_{m_1,m_2+n}
\endaligned
\end{equation}
for $n\in\{-2,-1,0,1,2\}$.

\section{Principal series for $SL(3,\R)$}\label{sec:princseries}

In complete analogy to (\ref{Vnudef}), principal series for $G=SL(3,\R)$ are defined as
\begin{multline}\label{princsl3}
  V_{\lambda,\delta} \ \ = \ \ \left\{f:G\rightarrow\C \,|\right. \\
 \left. f\(\tthree{a}{\star}{\star}{0}{b}{\star}{0}{0}{c}g\) \,=\, |a|^{1+\lambda_1}|b|^{\lambda_2}|c|^{-1+\lambda_3}\sgn(a)^{\delta_1}\sgn(b)^{\delta_2}\sgn(c)^{\delta_3}f(g)
\right\}
\end{multline}
for $\lambda=(\l_1,\l_2,\l_3)\in\C^3$ satisfying $\l_1+\l_2+\l_3=0$ and $\d=(\d_1,\d_2,\d_3)\in(\Z/2\Z)^3$; $\pi_{\l,\d}$
again acts by right translation.  The Iwasawa decomposition asserts that every element of $g$ is the product of an upper triangular matrix times an element of $K=SO(3,\R)$.  Thus all elements of $V_{\l,\d}$ are uniquely determined by their restrictions to $K$ and the transformation law in (\ref{princsl3}).  Just as was the case for $SL(2,\R)$ and $SO(2,\R)$, not all functions on $SO(3,\R)$ are restrictions of elements of $V_{\l,\d}$:~as before, the function must transform under any upper triangular matrix in $K$ according to the character defined in (\ref{princsl3}).

\begin{lem}\label{lem:so3parity}
Recall the Euler angles defined in (\ref{eulerangles}).
 A function $f:SO(3,\R)\rightarrow\C$ is the restriction of an element of $V_{\l,\d}$ if and only if
\begin{equation}\label{so3parityconditions}
\aligned
    f(k(\a,\b,\g)) \ \ & = \ \ (-1)^{\d_1+\d_2}\,f(k(\a+\pi,\b,\g)) \\ \text{and} \ \ \ \   f(k(\a,\b,\g)) \ \ & = \ \ (-1)^{\d_2+\d_3}\,f(k(\pi-\a,\pi-\b,\pi+\g))\,.
\endaligned
\end{equation}
\end{lem}
\begin{proof}
Consider the matrices
\begin{equation}\label{order4group}
 m_{-1,-1,1}  \ \  := \ \    \tthree{-1}000{-1}0001  \ \ \ \text{and} \ \ \
  m_{1,-1,-1}   \ \ :=  \ \   \tthree 1000{-1}000{-1} .
\end{equation}
Direct calculation shows that
\begin{equation}\label{m11kabgs}
\aligned
    m_{-1,-1,1}\,k(\a,\b,\g) \ \ & = \ \ k(\a+\pi  ,\b,\g) \\
\text{and} \ \ \
m_{1,-1,-1}\,k(\a,\b,\g) \ \ & = \ \ k(\pi-\a ,\pi-\b,\pi+\g )
\endaligned
\end{equation}
for any $\a,\g\in\R/(2\pi \Z)$ and $0\le \b\le \pi$.
Thus (\ref{so3parityconditions}) must hold for  functions $f\in V_{\l,\d}$.  Conversely,
since $m_{-1,-1,1}$ and $m_{1,-1,-1}$ generate all four upper triangular matrices in $K$,
  an extension of $f$ from $K=SO(3,\R)$ to $G=SL(3,\R)$ given by the transformation law in (\ref{princsl3}) is well-defined if it satisfies (\ref{so3parityconditions}).
\end{proof}

Like all functions on $K=SO(3,\R)$, $D^{\ell}_{m_1,m_2}$ has a unique extension   to
$G=SL(3,\R)$ satisfying the transformation law
\begin{multline}\label{WignerDonG}
  D^{\ell}_{m_1,m_2}\( \tthree{a}{\star}{\star}{0}{b}{\star}{0}{0}{c}k(\a,\b,\g)\) \ \ = \\
  a^{1+\lambda_1}b^{\lambda_2}c^{-1+\lambda_3} \,
  e^{i m_1 \a+i m_2 \g}\,d^{\ell}_{m_1,m_2}(\cos \b)\,, \ \ a,b,c\,>\,0\,.
\end{multline}
However, this extension may not be a well-defined element of the principal series $V_{\lambda,\delta}$ (\ref{princsl3}); rather, it is an element of the line bundle (\ref{linebundle}). Instead, certain linear combinations must be taken in order to account for the parities $(\d_1,\d_2,\d_3)$:

\begin{lem}\label{sl3Kmultiplicities}
The $V_\ell$-isotypic component of the principle series $V_{\l,\d}$ is spanned by
\begin{equation}\label{basisforlisotypicprincseries}
  \{ v_{\ell,m_1,m_2} \   | \  \srel{m_1\,\equiv\, \d_1+\d_2\imod 2}{-\ell\,\le\,m_2\,\le\,\ell} \}\,, \ \ v_{\ell,m_1,m_2} \ := \   D^{\ell}_{m_1,m_2}
  +  (-1)^{\d_1+\d_3+\ell}\, D^{\ell}_{-m_1,m_2}\,,
\end{equation}
where as always the subscripts $m_1$ and $m_2$ are integers satisfying the inequality $-\ell \le m_1,m_2\le \ell$.
In particular, $V_\ell$ occurs in $V_{\l,\d}$ with multiplicity
\begin{equation}\label{dldmult}
    m_{\ell,\d} \ \ = \ \ \left\{
                               \begin{array}{ll}
                                     \lfloor \smallf{\ell+1}{2} \rfloor, &  \d_1+\d_2 \ \text{odd,} \\
 1+\lfloor \smallf{\ell}{2} \rfloor, & \d_1+\d_3+\ell\ \text{even}\,, \ \d_1+\d_2 \ \text{even,} \\

                                  \lfloor \smallf{\ell}{2} \rfloor, & \d_1+\d_3+\ell\ \text{odd}\,,\  \d_1+\d_2 \ \text{even.} \\

                               \end{array}
                             \right.
\end{equation}
\end{lem}
 \begin{proof}
By \lemref{lem:so3parity} it suffices to determine which linear combinations of Wigner functions obey the two conditions in (\ref{so3parityconditions}).  The transformation properties of (\ref{WignerDonG})   show that the first condition is equivalent to the congruence  $m_1\equiv \d_1+\d_2\imod 2$.  The expression (\ref{littleddef}) shows that $d^{\ell}_{-m_1,m_2}(x)=(-1)^{\ell+m_2}d^{\ell}_{m_1,m_2}(-x)$, from which one readily sees the compatibility of the second condition in (\ref{so3parityconditions}) with the  sign $(-1)^{\d_1+\d_3+\ell}$ in (\ref{basisforlisotypicprincseries}).
\end{proof}

\noindent {\bf Examples:}
\begin{itemize}
\item
If $\ell=0$ the basis in (\ref{basisforlisotypicprincseries}) is nonempty if and only if $\d_1\equiv \d_2\equiv \d_3\imod 2$, which is the well-known criteria  for the existence of a spherical (i.e., $K$-fixed) vector in $V_{\l,\d}$.  The possible cases here are thus $(\d_1,\d_2,\d_3)\equiv (0,0,0)$ or $(1,1,1)\imod 2$, which are actually equivalent  because they are related by tensoring with the sign of the determinant character (since it is trivial on $SL(3,\R)$).

\item
If $\ell=1$ and $\d_1+\d_2\equiv 0\imod 2$, then $m_1$ must vanish and hence $\d_1+\d_3\equiv 1\imod 2$ in order to have a nonempty basis.  The possible cases of signs are $(\d_1,\d_2,\d_3)=(0,0,1)$ or $(1,1,0)$, which are again equivalent.
\end{itemize}

\section{The $({\frak g},K)$ module structure}\label{sec:gKmodulestructure}

It is a consequence of the Casselman embedding theorem \cite{casselman} that any irreducible representation of $G=SL(3,\R)$ is contained in some   principal series representation $V_{\l,\d}$ (\ref{princsl3}).  The Harish-Chandra module of $V_{\l,\d}$, its vector subspace of $K$-finite vectors, was seen in \lemref{sl3Kmultiplicities} to be isomorphic to $\oplus_{\ell\ge 0}V_{\ell}^{m_{\ell,\d}}$, where each copy of $V_{\l,\d}$ is explicitly indexed by certain  integers $\ell, m_1 \ge 0$ described in (\ref{basisforlisotypicprincseries}), and the multiplicity $m_{\ell,\delta}$ is given in (\ref{dldmult}).  An arbitrary subrepresentation $V$ of $V_{\l,\d}$  has a Harish-Chandra module isomorphic to $\oplus_{\l\ge 0}V_\ell^{m(V,\ell)}$, where $0\le m(V,\ell)\le m_{\ell,\d}$.

In \secref{sec:SL2R} we studied the Harish-Chandra modules of representations of $SL(2,\R)$ and their Lie algebra actions; in this section we consider these for $G=SL(3,\R)$.  Let us first denote some elements of the complexified Lie algebra ${\frak g}={\frak{sl}}_3(\C)$ of $G$ as follows:
\begin{equation}\label{sl3liealg}
\gathered
X_{1} \ \ = \ \ \liethree 010000000 \! , \ \ X_{2} \ \ = \ \ \liethree 000001000 \!   , \ \ X_{3} \ \ = \ \ \liethree 001000000\!, \\
X_{-1} \ \ = \ \ \liethree 000100000 \!  , \ \ \ X_{-2} \ \ = \ \ \liethree 000000010 \!  , \ \ X_{-3} \ \ = \ \ \liethree 000000100\!, \\
H_1 \ \ = \ \ \liethree 1000{-1}0000  \! , \ \ H_2 \ \ = \ \ \liethree 00001000{-1}\!,\\
Y_1 \ \ = \ \ -X_1 + X_{-1} \ \ = \ \ \liethree 0{-1}0100000 \!   , \ \
Y_2 \ \ = \ \ -X_2 + X_{-2} \ \ = \ \ \liethree 00000{-1}010\! , \\
Y_3 \ \ = \ \ -X_3 + X_{-3} \ \ = \ \ \liethree 00{-1}000100 \!, \\
Z_{-2} \ \ = \ \ \liethree 1i0i{-1}0000\! , \ \
Z_{-1} \ \ = \ \ \liethree 00i00{-1}i{-1}0\! , \ \
Z_{0} \ \ = \ \ \sqrt{\smallf 23}\liethree 10001000{-2}\! , \\
Z_{1} \ \ = \ \ \liethree 00i001i10\! , \ \
\text{and} \ \
Z_{2} \ \ = \ \ \liethree 1{-i}0{-i}{-1}0000\!.
\endgathered
\end{equation}
The normalization factor of $\sqrt{\f 23}$ for $Z_0$ is included to simplify later formulas.
The elements $\{X_1,X_2,X_3,X_{-1},X_{-2},X_{-3},H_1,H_2\}$ form  a basis of $\frak{sl}_3(\C)$.
  The elements $\{Y_1,Y_2,Y_3\}$ form a basis of $\frak k=\frak{so}_3(\C)$, which extends to the basis
\begin{equation}\label{convenientLiebasis}
\{Y_1,Y_2,Y_3,Z_{-2},Z_{-1},Z_0,Z_1,Z_2\}
\end{equation}
of $\frak{sl}_3(\C)$, in which the last 5 elements form a basis of the orthogonal complement $\frak p$ of $\frak k$ under the Killing form.  The elements $Z_j$ have been chosen so that
\begin{equation}\label{Zjproperty}
    \tthree{\cos \theta}{-\sin\theta}{0}{\sin\theta}{\cos\theta}{0}001 Z_j \tthree{\cos \theta}{-\sin\theta}{0}{\sin\theta}{\cos\theta}{0}001\i \ \ = \ \ e^{ij\theta}\,Z_j\,;
\end{equation}
that is, they diagonalize the adjoint action of the common $SO(2,\R)$-subgroup corresponding to the Euler angles $\alpha$ and $\gamma$.

 The rest of this section concerns  explicit formulas for the action of the basis (\ref{convenientLiebasis}) as differential operators under  right translation by $\pi$  as in  (\ref{liederiv}).  The formulas for differentiation by the first three elements, $Y_1,Y_2,Y_3$, are classical and are summarized as follows along with their action on Wigner functions  \cite{bied,NIST}.
In terms of the Euler angles (\ref{eulerangles}) on $K=SO(3,\R)$,
\begin{equation}\label{so3differentiation}
  \aligned
  \pi(Y_1) \ \ & = \ \ \f{\partial}{\partial \g} \,,\\
  \pi(Y_2) \ \ & = \ \ \f{\sin \g}{\sin \b}\f{\partial}{\partial \a} \ + \ \cos(\g)\,\f{\partial}{\partial \b} \ - \ \f{\sin \g}{\tan \b}\f{\partial}{\partial \g}\,,\\
  {\text{and}} \ \
  \pi(Y_3) \ \ & = \ \ -\,\f{\cos \g}{\sin\b}\f{\partial}{\partial \a} \ + \ \sin(\g) \,\f{\partial}{\partial \b} \ + \ \f{\cos\g}{\tan \b} \,\f{\partial}{\partial \g}\,.
  \endaligned
\end{equation}
The action of the differential operators (\ref{so3differentiation}) on the basis of Wigner functions $D^{\ell}_{m_1,m_2}$ is given by
\begin{equation}\label{so3diffonWignerD}
\aligned
  \pi(Y_1)D^{\ell}_{m_1,m_2} \ \ & = \ \  i m_2 D^{\ell}_{m_1,m_2}
  \\
  \pi(Y_2+iY_3)D^{\ell}_{m_1,m_2} \ \ & = \ \   \sqrt{\ell(\ell+1)-m_2(m_2+1)}\,D^{\ell}_{m_1,m_2+1}
   \\
    \text{and} \ \ \  \pi(-Y_2+iY_3)D^{\ell}_{m_1,m_2} \ \ & = \ \   \sqrt{\ell(\ell+1)-m_2(m_2-1)}\,D^{\ell}_{m_1,m_2-1}\,,
    \endaligned
\end{equation}
very much in analogy to the raising and lowering actions in (\ref{sl2updownaction}).
Here we recall that $\pi(Y_j)D^{\ell}_{m_1,m_2}(k):=\left. \f{d}{dt}\right|_{t=0} D^{\ell}_{m_1,m_2}(ke^{tY_j})$.  In terms of differentiation by left translation $L(Y_j)D^{\ell}_{m_1,m_2}(k):=\left. \f{d}{dt}\right|_{t=0} D^{\ell}_{m_1,m_2}(e^{tY_j}k)$,
\begin{equation}\label{leftderiv}
  \aligned
  	L(Y_1) D^\ell_{m_1,m_2}  \ \ = \ \ & i m_1 D^\ell_{m_1,m_2} \\
	L(-Y_2+i Y_3) D^\ell_{m_1,m_2}  \ \ = \ \ & \sqrt{\ell(\ell+1)-m_1(m_1+1)} D^\ell_{m_1+1,m_2} \\
	\text{and} \ \ \ L(Y_2+i Y_3) D^\ell_{m_1,m_2} \ \  = \ \ & \sqrt{\ell(\ell+1)-m_1(m_1-1)} D^\ell_{m_1-1,m_2}\,.
\endaligned
\end{equation}
This completely describes the Lie algebra action of $\frak k={\frak{so}_3}(\C)$ on the basis (\ref{basisforlisotypicprincseries}) of the Harish-Chandra module for $V_{\lambda,\delta}$.

We now turn to the key calculation of  the full Lie algebra action.  These formulas will be insensitive to the value of the parity parameter $\delta$ in the definition of the principal series (\ref{princsl3}).  For that reason, we will perform our calculations in the setting of the line bundle
\begin{equation}\label{linebundle}
  {\mathcal L}_{\l} \ \ = \ \ \left\{f:G\rightarrow\C \,|\,
  f\(\tthree{a}{\star}{\star}{0}{b}{\star}{0}{0}{c}g\) \,=\, a^{1+\lambda_1}b^{\lambda_2}c^{-1+\lambda_3}f(g), \, a,b,c>0
  \right\},
\end{equation}
which contains   $V_{\lambda,\delta}$ for any possible choice of $\delta$.  Elements of ${\mathcal L}_{\l}$ can be identified with their restrictions to $SO(3)$, and so for the rest of this section we shall tacitly identify each Wigner function $D^{\ell}_{m_1,m_2}$ with its extension to $G$ in   ${\mathcal L}_{\l}$ given in (\ref{WignerDonG}).  The right translation action $\pi$ on $V_{\lambda,\delta}$ also extends to ${\mathcal L}_{\l}$, which enables us to study the Lie algebra differentiation directly on $D^{\ell}_{m_1,m_2}$; the action on the basis elements (\ref{basisforlisotypicprincseries})  of $V_{\lambda,\delta}$ will follow immediately from this.  Though the passage to the line bundle $\mathcal L_\l$ is not completely necessary, it results in simpler formulas.

Given $X\in \frak g$ and $k\in K$, write
\begin{equation}\label{kXki}
  kXk^{-1} \ \ = \ \ X_{\frak{n}}(k) \ + \ X_{\frak{a}}(k) \ + \ X_{\frak{k}}(k)\,,
\end{equation}
 with $X_{\frak{n}}\in \frak{n}=\C X_1\oplus \C X_2\oplus \C X_3$,
$X_{\frak{a}}\in \frak{a}=\C H_1\oplus \C H_2$, and $X_{\frak{k}}\in \frak{k}=\C Y_1\oplus \C Y_2\oplus \C Y_3$.
Since $f(ke^{tX})=f(e^{tX_{\frak{n}}(k)+tX_{\frak{a}}(k)+tX_{\frak{k}}(k)}k)$, the derivative of this expression at $t=0$ is equal to
\begin{equation}\label{liealgdiff5}
  [\pi(X)f](k) \ = \   \left. \f{d}{dt}\right|_{t=0}f(e^{tX_{\frak{n}}(k)}k) \ + \  \left. \f{d}{dt}\right|_{t=0}f(e^{tX_{\frak{a}}(k)}k)
  \ + \  \left. \f{d}{dt}\right|_{t=0}f(e^{tX_{\frak{k}}(k)}k)\,.
\end{equation}
Write $X_{\frak{k}}=b_1(k)Y_1+b_2(k)Y_2+b_3(k)Y_3$ and  $X_{\frak{a}}=c_1(k)H_1+c_2(k)H_2$.
Since $f\in V_{\lambda,\d}$ satisfies the transformation law (\ref{linebundle}),
\begin{equation}\label{liealgdiff3a}
    \left. \f{d}{dt}\right|_{t=0}f(e^{tX_{\frak n}}g)f\ \ \equiv \ \ 0 \,,
\end{equation}
while
\begin{multline}\label{liealgdiff3b}
   \left. \f{d}{dt}\right|_{t=0}f(e^{tH_1}g) \ \ = \ \ (\l_1-\l_2+1)f(g)  \\
  \text{and} \ \
    \left. \f{d}{dt}\right|_{t=0}f(e^{tH_2}g) \ \ = \ \ (\l_2-\l_3+1)f(g)\,.
\end{multline}
Combining this with (\ref{leftderiv}),
we conclude
\begin{multline}\label{liealgdiff4}
  [\pi(X)D^{\ell}_{m_1,m_2}](k) \\ = \ \
 \( c_1(k)(\l_1-\l_2+1)+c_2(k)(\l_2-\l_3+1)+i m_1 b_1(k) \)D^{\ell}_{m_1,m_2}(k)
  \\
  - \ (b_2(k)+ib_3(k))\smallf{\sqrt{\ell(\ell+1)-m_1(m_1+1)}}{2}
  D^{\ell}_{m_1+1,m_2}(k)
  \\
  +\ (b_2(k)-ib_3(k))\smallf{\sqrt{\ell(\ell+1)-m_1(m_1-1)}}{2}
  D^{\ell}_{m_1-1,m_2}(k)\,,
\end{multline}
for any $\ell\ge 0$ and $-\ell\le m_1,m_2\le \ell$.

 Like all functions on $K$, each of the functions $c_1(k)$, $c_2(k)$, $b_1(k)$, $b_2(k)$, and $b_3(k)$ can be expanded as linear combinations of Wigner functions.  Applying the Clebsch-Gordan multiplication rule for products of two Wigner functions then exhibits (\ref{liealgdiff4}) as an explicit linear combination of Wigner functions.  We shall now compute these for the basis elements $X=Z_n$, $n\in \{-2,-1,0,1,2\}$, which of course entails no loss of generality:
 \begin{equation}\label{JBpf2}
\aligned
	c_1(Z_n) \ \ =& \ \    D^2_{-2,n}+\textstyle{\sqrt{\f 23}} \,D^2_{0,n} +   D^2_{2,n} \\
	c_2(Z_n)  \ \ =& \ \  2\,\textstyle{\sqrt{\f 23}}\, D^2_{0,n} \\
	b_1(Z_n) \ \ =& \ \ i   D^2_{-2,n}\,-\,i   D^2_{2,n} \\
	b_2(Z_n) \ \ =& \ \ -\!D^2_{-1,n}\,+\,D^2_{1,n} \\
	b_3(Z_n) \ \ =& \ \ i D^2_{-1,n}\,+\,i D^2_{1,n}\,,
\endaligned
\end{equation}
as can be checked via direct computation.
We now state the action of the $Z_n$ on Wigner functions:
\begin{thm}
\label{thm:main}
Let
\begin{equation}\label{Lambda}
    \aligned
\L^{(-2)}_j(\l,\ell,m_1) \ \ & = \ \ \l_1\,-\,\l_2\,+\,1\,-\,m_1 \\
\L^{(0)}_j(\l,\ell,m_1) \ \ & = \ \ \l_1\,+\,\l_2\,-\,2\l_3+j\ell+\f{j+j^2}{2} \\
\L^{(2)}_j(\l,\ell,m_1) \ \ & = \ \ \l_1\,-\,\l_2\,+\,1\,+\,m_1\,, \\
\endaligned
\end{equation}
 $c_{-2}=c_2=1$, $c_0=\sqrt{\f 23}$, and recall the formulas for $q_{k,j}(\ell,m)=\langle 2k\ell m|(\ell+j)(k+m)\rangle$ given in (\ref{CG3}). For $n\in \{-2,-1,0,1,2\}$,
\begin{equation}\label{liealgdiffZj}
    \pi(Z_n)D^{\ell}_{m_1,m_2}  \ \ =   \sum_{\srel{-2\le j \le 2}{k\,\in\,\{-2,0,2\}}} c_k \,  q_{k,j}(\ell,m_1)\, q_{n,j}(\ell,m_2) \, \Lambda^{(k)}_j(\lambda,\ell,m_1) \, D^{\ell+j}_{m_1+k,m_2+n}
\end{equation}
as an identity of elements in the line bundle $\mathcal L_\l$  from  (\ref{linebundle}).
\end{thm}
\begin{proof}
Formulas (\ref{liealgdiff4}) and (\ref{JBpf2}) combine to show
\begin{equation}\label{JB3}
  \aligned
  	\pi(Z_n) D^\ell_{m_1,m_2}  \ \ =& \ \ c_{-2} (\lambda_1-\lambda_2+1-m_1) D^2_{-2,n} D^\ell_{m_1,m_2} \\
	&+c_{0} (\lambda_1+\lambda_2-2\lambda_3+3) D^2_{0,n} D^\ell_{m_1,m_2} \nonumber \\
	&+c_{2} (\lambda_1-\lambda_2+1+m_1) D^2_{2,n} D^\ell_{m_1,m_2} \nonumber \\
	&+\sqrt{\ell(\ell+1)-m_1(m_1+1)} D^2_{-1,n} D^\ell_{m_1+1,m_2} \nonumber \\
	&+\sqrt{\ell(\ell+1)-m_1(m_1-1)} D^2_{1,n} D^\ell_{m_1-1,m_2}. \nonumber
  \endaligned
\end{equation}
The Theorem now follows from (\ref{CG1}) and (\ref{CG5}).

\end{proof}

Theorem~\ref{thm:intro} follows immediately from Theorem~\ref{thm:main}, the definition of $v_{\ell,m_1,m_2}$ given in (\ref{basisforlisotypicprincseries}), and the identity $q_{k,j}(\ell,m)=(-1)^j q_{-k,j}(\ell,-m)$.  Formula (\ref{liealgdiffZj}) 
expresses Lie algebra derivatives of $D^{\ell}_{m_1,m_2}$ as linear combinations of  $V_{\ell+j}$-isotypic vectors for $-2\le j\le 2$.  We shall now explain how the operators $U_j$ defined by
\begin{align}\label{Ujdef}
	U_j D^\ell_{m_1,m_2} \ \  = \ \ & \sum_{k\,\in\,\{-2,0,2\}} c_k \, q_{k,j}(\ell,m_1) \, \Lambda^{(k)}_j(\lambda,\ell,m_1)\, D^{\ell+j}_{m_1+k,m_2}
\end{align}
for $-2\le j \le 2$ can usually be written using Lie algebra differentiation under $\pi$.  These map $V_\ell$-isotypic vectors to $V_{\ell+j}$-isotypic vectors, and thus
 separate out the contributions to (\ref{liealgdiffZj}) for fixed $j$ (aside from an essentially harmless shift of $m_2$).  Since they do not require linear combinations, the operators $U_j$ are in a sense more analogous to Maass's  raising and lowering operators (\ref{sl2updownaction}), and often more useful than the $\pi(Z_n)$.

To write $U_j$ in terms of Lie algebra derivatives, let
\begin{align}
	W^\ell_{\pm 2,m_2}  \ \ & = \ \   \pi(\smallf{\mp Y_2+i Y_3}{\sqrt{\ell(\ell+1)-m_2(m_2\pm1)}}) \circ
\pi( \smallf{\mp Y_2+i Y_3}{\sqrt{\ell(\ell+1)-(m_2\pm1)(m_2\pm2)}}) \circ \pi( Z_{\pm 2})\,, \\
	 W^\ell_{\pm 1,m_2}  \ \  & =  \ \  \pi\(\frac{\mp Y_2+i Y_3}{\sqrt{\ell(\ell+1)-m_2(m_2\pm1)}} \) \circ \pi(Z_{\pm 1}) \,, \ \ \text{and}\\
 W^\ell_{0,m_2}   \ \  & =    \ \   \pi(Z_0)\,,
\end{align}
which are defined whenever the arguments of the square-roots are positive.
Let $\Delta_K$ be the $SO(3)$ laplacian, which acts on Wigner functions by $\D_K D^\ell_{m_1,m_2}=\ell(\ell+1)D^\ell_{m_1,m_2}$ \cite[Section 3.8]{bied}.  Then the (commutative) composition of operators
\begin{align}
	P^\ell_j \ \  =  \ \ \prod_{\substack{|k| \le 2\\ k\ne j \\ \ell+k \ge 0}} \frac{\Delta_K- (\ell+k)(\ell+k+1)}{(\ell+j)(\ell+j+1)-(\ell+k)(\ell+k+1)}
\end{align}
 acts on Wigner functions by the formula
\begin{equation}\label{Plproject}
  P^\ell_jD^{\ell'}_{m_1,m_2} \ \ = \ \ \left\{
                                          \begin{array}{ll}
                                            0, & \ell'\neq \ell+j \\
                                            D^{\ell+j}_{m_1,m_2}, &\ell'=\ell+j
                                          \end{array}
                                        \right.
\end{equation}
for $\ell-2\le \ell' \le \ell+2$.  After composing with this projection operator,
it follows from (\ref{so3diffonWignerD}) and (\ref{liealgdiffZj}) that
\begin{equation}\label{PWqU}
  P^\ell_{j}\circ W^{\ell+j}_{n,m_2}  \ \ = \ \  q_{n,j}(\ell,m_2)\, U_j
\end{equation}
on the span of Wigner functions $D^{\ell}_{m_1,m_2}$ with $|m_1|\le\ell$.
Furthermore,  for any choice of $\ell,m_2,j$ for which $(\ell,j)\not\in \{(0,0),(0,1),(1,-1)\}$,
  there is some $-2\le n \le 2$ with  $q_{n,j}(\ell,m_2) \ne 0$, so that $U_j$ coincides with  the Lie algebra differentiation $q_{n,j}(\ell,m_2)^{-1} P^\ell_{j}\circ W^{\ell+j}_{n,m_2}$ on this span.

\section{Examples of composition series}\label{sec:examples}

In this section we present a selected assortment of  examples of representations of $SL(3,\R)$ that are related to automorphic forms, and explicitly describe their $K$-type structure in terms of Wigner functions.  The treatment here is by no means exhaustive.  However, the techniques we use are directly applicable to any representation of $SL(3,\R)$.
We also briefly explain how to compute the composition series in some examples, though for convenience we incorporate information from the {\sc atlas} software package \cite{atlas} (which computes the length of the composition series as well as the multiplicities of the $K$-types of each factor).  A similar analysis was performed in  \cite{howe} using a different model for principal series.

It is also possible to use the theory of intertwining operators to obtain descriptions of invariant subspaces in terms of Wigner functions.  Indeed, the simple intertwining operator corresponding to the Euler angle $\alpha$ in (\ref{eulerangles}) is diagonalized by Wigner functions, via what amounts to an $SL(2)$-calculation.  The simple intertwining operator corresponding to the Euler angle $\beta$ does not act diagonally on this basis, but does act diagonally on a basis of Wigner-like functions defined instead  using different Euler angles related by conjugation by a Weyl element.  Using a permutation matrix in $K=SO(3,\R)$ along with the fact that the Wigner $D$-matrix $(D^\ell_{m_1,m_2}(k))_{-\ell\le m_1,m_2\le \ell}$  is a representation of $K$, it is trivial to explicitly diagonalize this intertwining operator as well (but not both intertwining operators simultaneously).

\subsection{ Representations induced from trivial on $SL(2,\R)$}\label{sec:sub:inducedtriv}

We shall now describe how  \thmref{thm:main} recovers simpler, known formulas for some degenerate principal series representations (see
\cite{howelee}).
Let $\pi=\pi_s$ denote the subspace of $V_{(s-1/2,s+1/2,-2s),(0,0,0)}$ spanned by $D^{\ell}_{0,m_2}$, $|m_2|\le \ell$, with $\ell$ even (a parity condition forced by  (\ref{basisforlisotypicprincseries})).
Thus the underlying Harish-Chandra module of $\pi$ is
\begin{equation}\label{inducedfromtriv1}
  HC(\pi) \ \ = \ \ \bigoplus_{\srel{\ell\, \ge \,0}{\ell \,\in\,2\Z}} V_{\ell}
\end{equation}
and each $V_{\ell}$ has basis $\{D^{\ell}_{0,m_2}|-\ell \le m_2 \le \ell\}$.

It is not hard to see that
\begin{equation}\label{indfromtrivial1}
  \pi \ \ = \ \ \text{Ind}_{P_{2,1}}^G\psi_s
\end{equation}
 is  induced from a quasicharacter $\psi_s$ of a maximal parabolic subgroup $P_{2,1}$ of $G=SL(3,\R)$, from which it follows that it is a $SL(3,\R)$-invariant subspace of $V_{(s-1/2,s+1/2,-2s),(0,0,0)}$ (equivalently, that (\ref{inducedfromtriv1}) is ${\frak{sl}}(3,\R)$-invariant).  The formulas in \secref{sec:gKmodulestructure} give much finer information.
With $(\l_1,\l_2,\l_3)=(s-1/2,s+1/2,-2s)$ and $m_1=0$, (\ref{Lambda}) reads
\begin{equation}\label{Lambdadegen}
    \aligned
\L^{(-2)}_j(\l,\ell,m_1) \ \ &  = \ \ 0 \\
\L^{(0)}_j(\l,\ell,m_1) \ \ &   = \ \ 4s +j\ell+\smallf{j+j^2}{2}  \\
\L^{(2)}_j(\l,\ell,m_1) \ \ &  = \ \ 0 \,.\\
\endaligned
\end{equation}
Note also that $D^{\ell}_{0,m_2}$'s extension (\ref{WignerDonG}) is a well-defined element of $V_{\l,\d}$ for $\ell$ even.
Thus, in this situation formula (\ref{liealgdiffZj}) from \thmref{thm:main} states
\begin{equation}\label{liealgdiffZjdegen12}
\aligned
    \pi(Z_n)D^{\ell}_{0,m_2}  \ \ & = \ \  \sqrt{\smallf 23}\,\sum_{j\,\in\,\{-2,0,2\}} q_{0,j}(\ell,0)\,q_{n,j}(\ell,m_2)\,(4s+j\ell+\smallf{j+j^2}{2})\,
D^{\ell+j}_{0,m_2+n}\,,\\
\endaligned
\end{equation}
since $q_{0,-1}(\ell,0)=q_{0,1}(\ell,0)=0$ by (\ref{CG3}).  Furthermore, (\ref{Ujdef}) specializes to the formula
\begin{equation}\label{liealgeUjdegen}
  U_j D^{\ell}_{0,m_2} \ \ = \ \  \sqrt{\smallf 23}\, q_{0,j}(\ell,0)\,(4s+j\ell+\smallf{j+j^2}{2})\,D^{\ell+j}_{0,m_2} \,,
\end{equation}
which vanishes if $j=\pm 1$.  Here $U_{\pm 2}$ play the role of the raising and lowering operators in (\ref{sl2updownaction}), from which the reducibility of $\pi$  can be easily deduced as it was for the principal series in \thmref{thm:irrepsofSL2R}.

\begin{figure}\label{fig:1}
\caption{An illustration of (\ref{liealgeUjdegen}) for  representations of $SL(3,\R)$ induced from the trivial representation of $SL(2,\R)$.  Each copy of $V_\ell$ is spanned by  $\{D^{\ell}_{0,m_2}|-\ell\le m_2\le \ell\}$.  At certain $s\in \frac{1}{4}\Z$ the operators $U_{\pm 2}$ may be trivial, in which case the representation reduces (analogously to the reducibility of the $SL(2,\R)$ principal series in Theorem~\ref{thm:irrepsofSL2R}).}
\begin{center}
\includegraphics[width=5.5in]{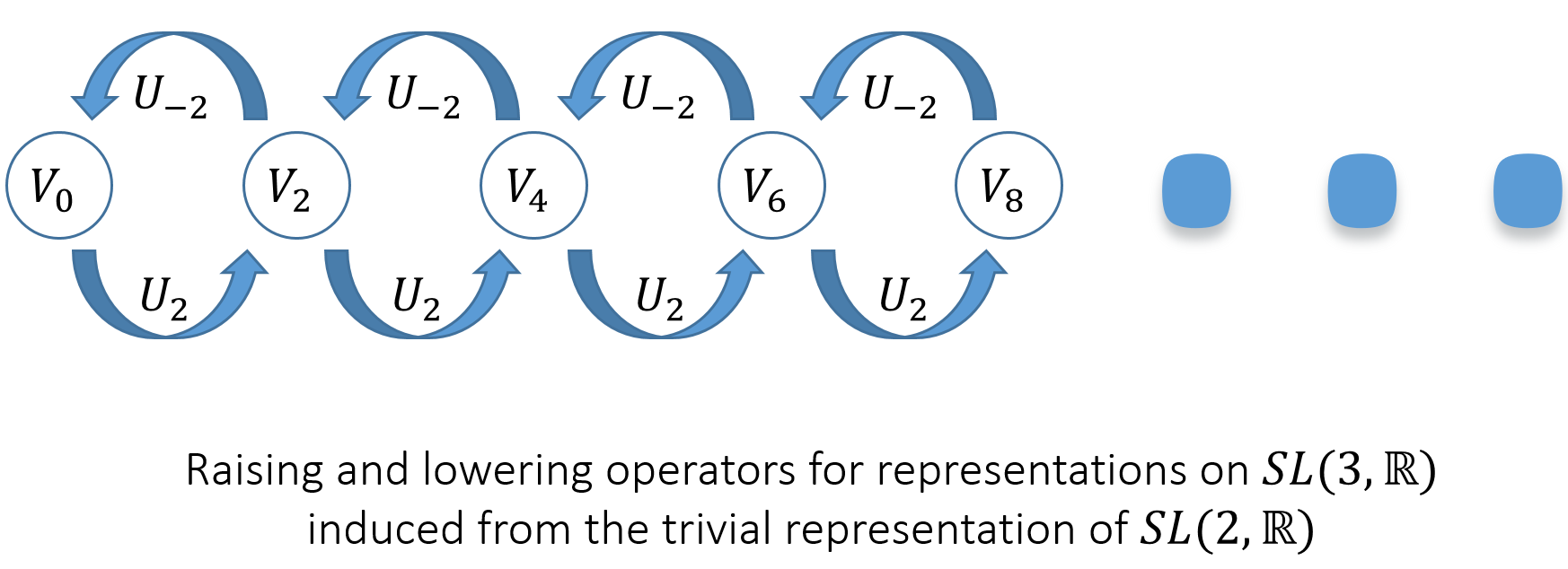}
\end{center}
\end{figure}

\subsection{Cohomological representations}

In the rest of this section we will consider the decomposition of the full principal series
\begin{equation}\label{fullprinck}
 V_{(-\f{k-1}{2},\f{k-1}2,0),(k,0,k)}\ , \ \ \ \ k\,\in\,\Z_{\ge 2}\,.
\end{equation}
  (The principal series $V_{(-\f{k-1}{2},\f{k-1}2,0),(0,k,k)}$ have a very similar analysis.)
Recall from (\ref{dldmult}) that   $V_\ell$ occurs with multiplicity $m_{\ell,\d}$, which for  $k$   even  is $1+\lfloor \f\ell2 \rfloor$ if $\ell$ is even and $ \lfloor \f\ell2 \rfloor$ if $\ell$ is odd; for $k$ odd $m_{\ell,\delta}=\lfloor \f{\ell+1}{2}\rfloor$.
The elements $
  v_{\ell,m_1,m_2}   =   D^\ell_{m_1,m_2}  +  (-1)^{\ell}  D^{\ell}_{-m_1,m_2}$ from (\ref{basisforlisotypicprincseries})
are nonzero elements of (\ref{fullprinck}) when $m_1\neq 0$ and $m_1 \equiv k\imod 2$;  $v_{\ell,0,m_2}\neq 0$ when $k$ and $\ell$ are even.

\begin{thm}\label{thm:evenk}  For even integers $k\ge 2$, the spherical representation $V_{(-\f{k-1}{2},\f{k-1}2,0),(0,0,0)}$ has 2 constituents, $V_A$ and $V_B$, which contain $V_\ell$ with the following multiplicities:
\begin{itemize}
  \item $m_{\ell,A}=\max(0,1+\lfloor \f{\ell-k}{2}\rfloor)$,   and
  \item $m_{\ell,B}=\left\{
                      \begin{array}{ll}
                        \min(\lfloor \f\ell2 \rfloor,\f{k-2}{2}), & \ell \text{ odd} \\
                        \min(1+\lfloor \f\ell2 \rfloor,\f k2), &  \ell \text{ even.}
                      \end{array}
                    \right.$
\end{itemize}
The representation $V_B$ is the subrepresentation spanned by the $v_{\ell,m_1,m_2}$  having $0\le m_1<k$ (and satisfying the above parity constraints), whereas $V_A$ is the quotient of $V_{(-\f{k-1}{2},\f{k-1}2,0),(0,0,0)}$ by $V_B$.  The dual representation $V_{-(-\f{k-1}{2},\f{k-1}2,0),(0,0,0)}$ contains $V_A$ as the subrepresentation spanned by the  $v_{\ell,m_1,m_2}$  having $  m_1\ge k$ (and satisfying the above parity constraints), and $V_B$ as a quotient.

The composition series for various Weyl orbits of $(-\f{k-1}{2},\f{k-1}{2},0)$ have the form
$\{0\} \subset V  \subset V_{\lambda,(0,0,0)}$ with:

\begin{itemize}
\item $V\cong V_A$ and $V_{\lambda,(0,0,0)}/V\cong V_B$ if $\lambda=(\f{k-1}{2},-\f{k-1}{2},0)$, $(0,\f{k-1}{2},-\f{k-1}{2})$, or $(\f{k-1}{2},0,-\f{k-1}{2})$; and
\item $V\cong V_B$ and $V_{\lambda,(0,0,0)}/V\cong V_A$ if $\lambda=(-\f{k-1}{2},\f{k-1}{2},0)$, $(0,-\f{k-1}{2},\f{k-1}{2})$, or $(-\f{k-1}{2},0,\f{k-1}{2})$.
\end{itemize}
\end{thm}

\noindent Note that when $k=2$, $V_B$ coincides with the $s=0$ case of Section~\ref{sec:sub:inducedtriv}.

\begin{figure}[t]\label{fig:2}
\caption{A schematic illustration of Theorem~\ref{thm:evenk}.}
\begin{center}
\includegraphics[width=5.3in]{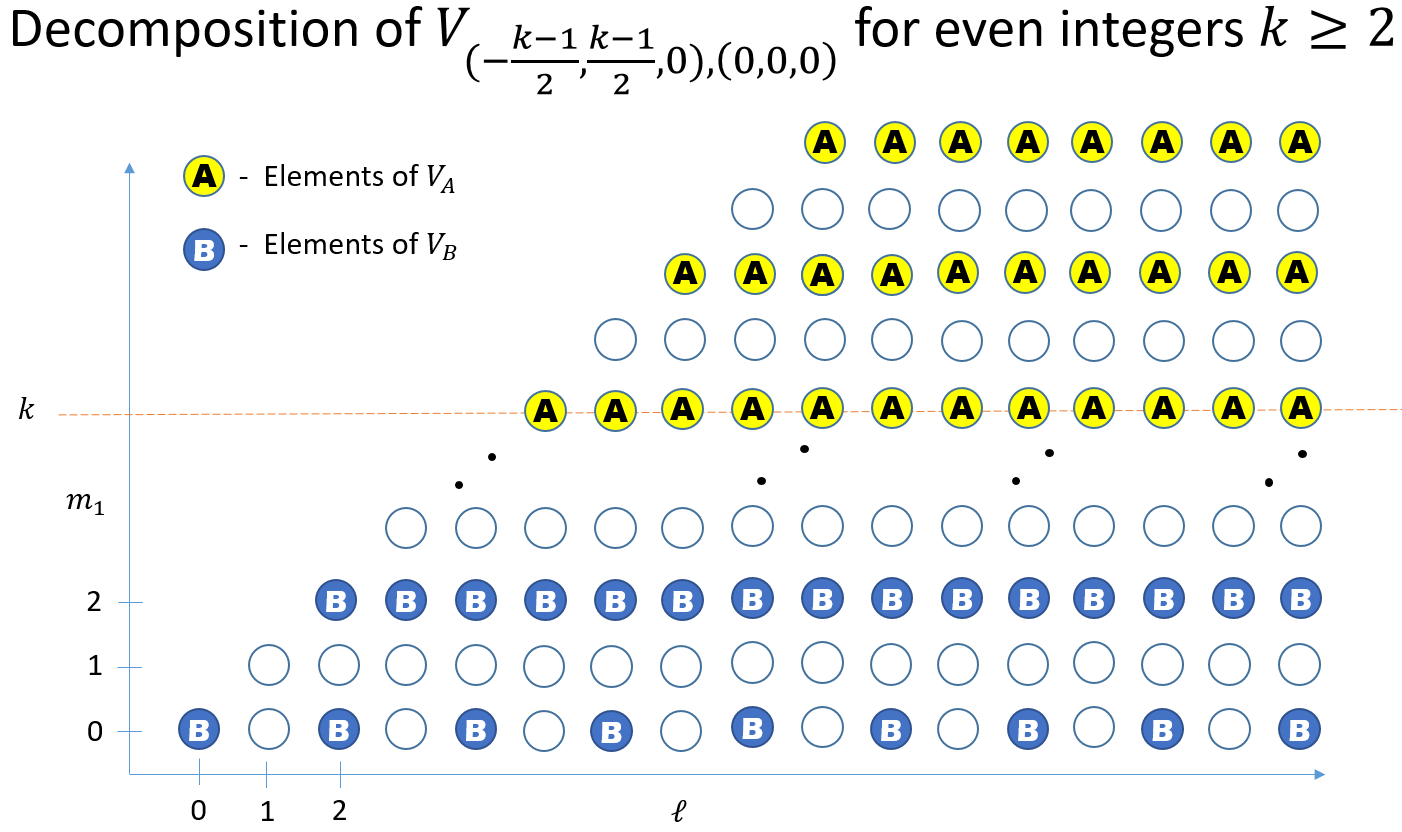}
\end{center}
\end{figure}

\begin{proof}
The fact that $V_{(-\f{k-1}{2},\f{k-1}2,0),(0,0,0)}$ has a composition series of length 2   for $k\in 2\Z_{\ge 1}$ is a standard consequence of the Beilinson-Bernstein theory.  Both this and the assertions about the composition series for $\l=\pm(\f{k-1}{2},0,-\f{k-1}{2})$
can be directly verified through it  using the {\tt atlas} software package \cite{atlas}.  Let $W$ be the subspace of $V_{(-\f{k-1}{2},\f{k-1}2,0),(0,0,0)}$ spanned by its basis vectors $v_{\ell,m_1,m_2}$ with $m_1<k$.
 Since $\L_j^{(2)}((-\f{k-1}{2},\f{k-1}2,0),\ell,k-2)$ and $\L_j^{(-2)}((-\f{k-1}{2},\f{k-1}2,0),\ell,2-k)$ in (\ref{Lambda}) both vanish, formula (\ref{liealgdiffZj}) shows that $\pi(Z_n)$ preserves $W$, which is consequently irreducible.
 Similarly, the subspace of $V_{-(-\f{k-1}{2},\f{k-1}2,0),(0,0,0)}$ spanned by its basis vectors $v_{\ell,m_1,m_2}$ with $m_1\ge k$ is also irreducible.
 The multiplicity formulas follow from this, as does  the composition series assertion for $\l=(-\f{k-1}{2},\f{k-1}2,0)$.  The remaining three   composition series assertions follow by duality and the contragredient symmetry.
\end{proof}

\subsection{Constant coefficients cuspidal cohomology}

We next turn to the representation (\ref{fullprinck}) for $k=3$,
which is spanned by the $v_{\ell,m_1,m_2}$ with $\ell>0$, $m_1$ odd, and $-\ell\le m_1,m_2 \le \ell$.  It has a composition series
\begin{equation}\label{constantcoeffcohom1}
  \{0\} \ \ \subset \ \ V^{(1)}_{\text{odd}} \ \ \subset \ \ V^{(1)} \ \ \subset \ \ V_{(-1,1,0),(1,0,1)}
\end{equation}
of length 3, where $V^{(1)}$ is spanned by the $v_{\ell,1,m_2}$ with $\ell>0$ and $V^{(1)}_{\text{odd}}$ is spanned by the $v_{\ell,1,m_2}$ with $\ell>0$ and odd.  As in Theorem~\ref{thm:evenk}, it follows from (\ref{Lambda})-(\ref{liealgdiffZj})  that $V^{(1)}_{\text{odd}}$ and $V^{(1)}/V^{(1)}_{\text{odd}}$ are irreducible.

 The quotient $V_{(-1,1,0),(1,0,1)}/V^{(1)}$ is the archimedean representation associated to constant coefficients automorphic cohomology (it is the symmetric square $\bigwedge\!D_2$ of the discrete series $D_2$ of $SL(2,\R)$).  By duality, $\bigwedge\!D_2$ is an irreducible subspace of $V_{(1,-1,0),(1,0,1)}$, which can be shown as above to be spanned by the $v_{\ell,m_1,m_2}$ with $\ell>0$ and  $m_1\ge 3$  odd.

\begin{figure}[t]\label{fig:3}
\caption{An illustration  the cohomological representation $\bigwedge\!D_2$    as a quotient of $V_{(-1,1,0),(1,0,1)}$.}
\begin{center}
\includegraphics[width=5.3in]{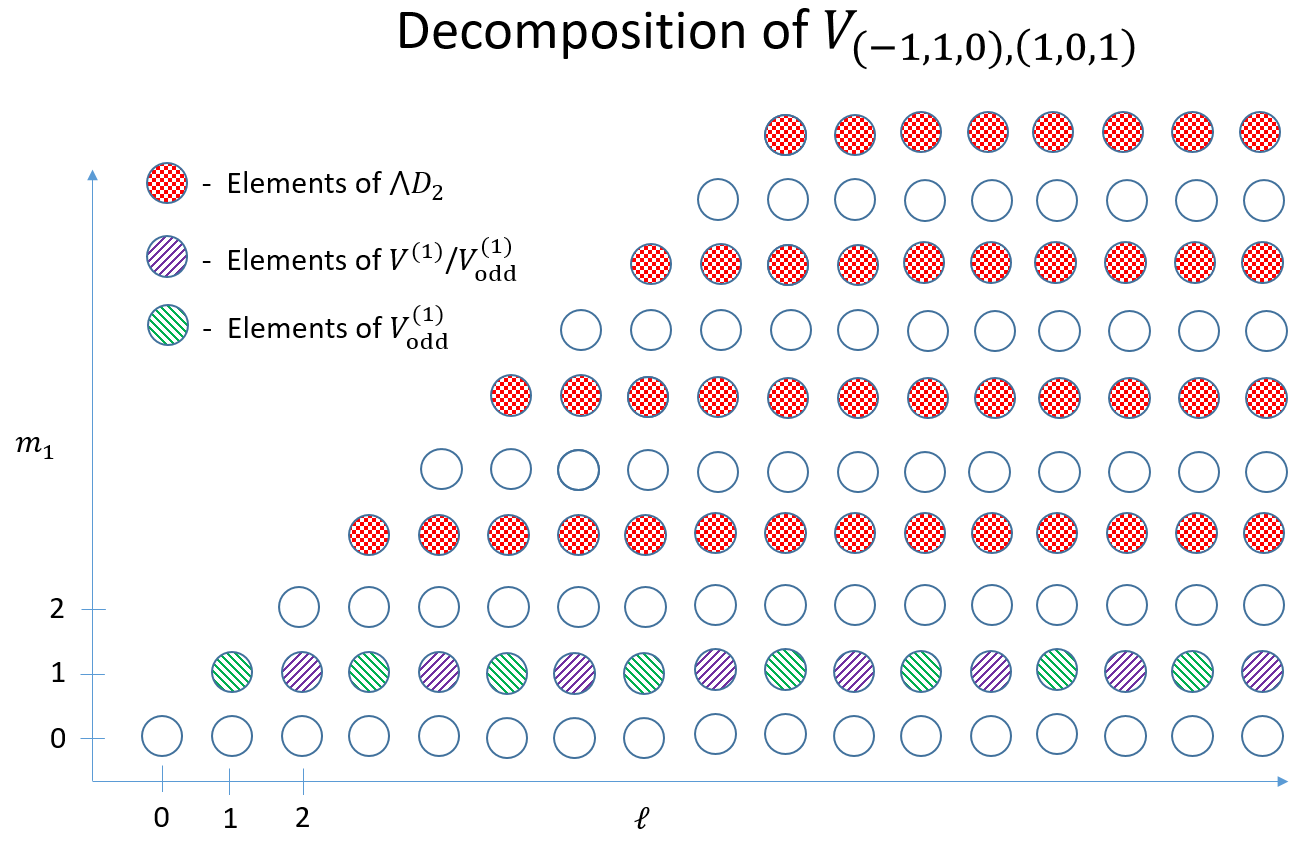}
\end{center}
\end{figure}

\subsection{Symmetric square of Ramanujan's cusp form $\Delta$}

Our final example is (\ref{fullprinck}) for $k=23$.  Its dual $V_{(11,-11,0),(1,0,1)}$ contains the irreducible representation $\bigwedge\!D_{12}$, which is the archimedean representation associated to the symmetric square of Ramanujan's famous weight 12 cusp form $\Delta$.  This subspace of $V_{(11,-11,0),(1,0,1)}$  is spanned by the $v_{\ell,m_1,m_2}$ with $\ell\ge 23$  and $m_1\ge 23$ odd.  The other composition factors in these principal series are more complicated to describe.

\section{Formulas for action as differential operators on $G/K$}
\label{sec:diffopformulas}

With applications of the previous results to automorphic forms and string theory in mind, it may be useful to write the action of the differential operators  (\ref{liealgdiffZj}) on functions on the symmetric space $SL(3,\R)/SO(3)$.  We first review this for $SL(2,\R)$ (where the results are well-known and classical).  In this section $\pi$ will be used to denote the right translation operator.

\subsection{$SL(2,\R)/SO(2)$}

Consider the commonly-used coordinate parametrization
\begin{multline}\label{sl2rcoordinates1}
    g \ \ = \ \ \ttwo1x01 \ttwo{y^{1/2}}{0}{0}{y^{-1/2}}\ttwo{\cos(\theta)}{-\sin(\theta)}{\sin(\theta)}{\cos\theta} \ \ \in \ \
SL(2,\R)\,, \\ x\,\in\,\R\,,\ y\,>\,0\,, \ \ \theta\,\in\,\R/(2\pi\Z)\,,
\end{multline}
of $G=SL(2,\R)$.   Let $f$ be an $SO(2)$-isotypic function on $G$, meaning it satisfies the right-transformation law
\begin{equation}\label{sl2coordinates2}
    f\(g\ttwo{\cos(\theta)}{-\sin(\theta)}{\sin(\theta)}{\cos(\theta)}\) \ \ = \ \ e^{i\ell\theta}\,f(g)\,,\ \text{for all} \ \theta\,\in\,\R\,,
\end{equation}
or, equivalently, the differential equation
\begin{equation}\label{sl2diffops1}
    \pi\ttwo{0}{-1}{1}{0}f \ \ = \ \ i\,\ell\, f
\end{equation}
for some $\ell\in\Z$.
Then $f$ is determined by its restriction   to upper triangular matrices $\ttwo1x01 \ttwo{y^{1/2}}{0}{0}{y^{-1/2}}$, and the action of the rest of the Lie algebra  on $f$ is determined by the formulas
\begin{equation}\label{sl2diffops2}
\aligned
    \(\pi\ttwo{1}{0}{0}{-1}\!f\)\!\(\ttwo1x01 \ttwo{y^{1/2}}{0}{0}{y^{-1/2}}\) \ \ & = \ \
2\,y\,\smallf{\partial}{\partial y}\,f\(\ttwo1x01 \ttwo{y^{1/2}}{0}{0}{y^{-1/2}}\) \\ \text{and} \ \ \ \
\(\pi\ttwo0100 \! f\)\!\(\ttwo1x01 \ttwo{y^{1/2}}{0}{0}{y^{-1/2}}\) \ \ & = \ \ y\,\smallf{\partial}{\partial x}\,f\(\ttwo1x01 \ttwo{y^{1/2}}{0}{0}{y^{-1/2}}\)
\endaligned
\end{equation}
via linear combinations.  For example,   the raising and lowering operators in (\ref{sl2updownaction}) correspond to the differential operators
\begin{equation}\label{sl2diffops3}
  \(  \pi\ttwo{1}{\mp i}{\mp i}{-1}\!f\)\!\(\ttwo1x01 \ttwo{y^{1/2}}{0}{0}{y^{-1/2}}\)  \  = \
(\pm 2iy\smallf{\partial}{\partial x}+  2y\smallf{\partial}{\partial y} \pm \ell )f\(\ttwo1x01 \ttwo{y^{1/2}}{0}{0}{y^{-1/2}}\)
\end{equation}
introduced by Maass.

\subsection{$SL(3,\R)/SO(3)$}

We parameterize elements of $G=SL(3,\R)$ as
\begin{equation}\label{sl3coordinates1}
    g \ \ = \ \ \tthree{1}{x_1}{x_3}01{x_2}001\tthree{y_1^{2/3}y_2^{1/3}}000{y_1^{-1/3}y_2^{1/3}}000{y_1^{-1/3}y_2^{-2/3}} k(\a,\b,\g) \ \ \in \ \ G\,,
\end{equation}
where $x_1,x_2,x_3\in \R$, $y_1,y_2>0$, and $k(\a,\b,\g)\in K=SO(3,\R)$ is defined in  (\ref{eulerangles}).
Assume that $f$ is an $SO(3)$-isotypic function on $G$ satisfying the transformation law
\begin{equation}\label{sl2coordinates2}
    f\(g\tthree{\cos(\theta)}{-\sin(\theta)}0{\sin(\theta)}{\cos(\theta)}0001\) \ \ = \ \ e^{im_2\theta}\,f(g)\,,\ \text{for all} \ \theta\,\in\,\R\,,
\end{equation}
which is the case for $D^{\ell}_{m_1,m_2}$ (see (\ref{WignerDdef})) and hence the image of a $K$-isotypic vector corresponding to $v_{\ell,m_1,m_2}$ in an irreducible representation of $G$.
We refer back to (\ref{sl3liealg}) for notation of Lie algebra elements.
\begin{lem}\label{lem:diffops}
Let $g(x_1,x_2,x_3,y_1,y_2)$ denote the product of the first two matrices in (\ref{sl3coordinates1}).  Then
\begin{equation}\label{diffopssl3a}
 ( \pi(X)f ) (g(x_1,x_2,x_3,y_1,y_2)) \ \  = \ \ \mathcal{D}f(g(x_1,x_2,x_3,y_1,y_2))
\end{equation}
for the following pairs of Lie algebra elements $X$ and differential operators $\mathcal{D}$:
\begin{equation}\label{diffopssl3b}
\begin{tabular}{|c|c|}
  \hline
 X & $\mathcal{D}$ \\
  \hline
  $Y_1$ & $im_2$ \\
  $H_1$ & $2y_1\f{\partial}{\partial y_1}-y_2\f{\partial}{\partial y_2}$ \\
  $H_2$ & $-y_1\f{\partial}{\partial y_1}+2y_2\f{\partial}{\partial y_2}$ \\
  $X_1$ & $y_1\f{\partial}{\partial x_1}$ \\
  $X_2$ & $y_2\f{\partial}{\partial x_2}+x_1y_2\f{\partial}{\partial x_3}$ \\
  $X_3$ & $y_1y_2\f{\partial}{\partial x_3}$ \\
  $Z_{-2}$ & $-m_2
+ 2iy_1\f{\partial}{\partial x_1}+2y_1\f{\partial}{\partial y_1}-y_2\f{\partial}{\partial y_2}$ \\
  $Z_0$ & $\sqrt{6}y_2\f{\partial}{\partial y_2}$ \\
  $Z_{2}$ & $m_2
- 2iy_1\f{\partial}{\partial x_1}+2y_1\f{\partial}{\partial y_1}-y_2\f{\partial}{\partial y_2}$ \\
  \hline
\end{tabular}
\end{equation}
\end{lem}

\noindent
Expressions for $Y_2$ and $Y_3$ (and hence $Z_{\pm 1}$) can also be given, but they are more complicated in the absence of an assumption such as  (\ref{sl2coordinates2}).

\end{document}